\documentclass[reqno,12pt]{amsart}
\usepackage{amssymb}
\usepackage{amsmath}
\usepackage{enumerate}
\usepackage{amsfonts}
\usepackage{color}
\usepackage{graphics}
\usepackage{comment}

\newtheorem{theorem}{Theorem}[section]
\newtheorem{lemma}[theorem]{Lemma}

\theoremstyle{remark}
\newtheorem{remark}[theorem]{Remark}

\theoremstyle{remark}
\newtheorem*{note}{Remark}

\theoremstyle{definition}
\newtheorem{definition}[theorem]{Definition}

\numberwithin{equation}{section}

\DeclareMathOperator{\supp}{supp}
\DeclareMathOperator{\Var}{Var}

\DeclareMathOperator{\sign}{sign}
\newcommand{\e}{\text{\bf E}}
\newcommand{\p}{\text{\bf P}}

\newcommand{\F}{\mathbb{F}}

\newcommand{\N}{\mathbb{N}}

\newcommand{\R}{\mathbb{R}}

\newcommand{\s}{\mathbb{S}}

\def\8{\infty}
\def\blab#1{\begin{equation}\label{#1}}
\def\elab{\end{equation}}

\def\blab#1{\begin{equation}\label{#1}}
\def\elab{\end{equation}}

\title[Fuzzy random walks]{Convergence of fuzzy random walks to a standard Brownian motion}

\author[J. Schneider]{Jan Schneider}
\address{Institute of Industrial Engineering and Management
at the Faculty of Computer and Management Science\\
Wroclaw University of Technology\\
ul.
{\L}ukasiewicza 3/5\\
50-372 Wroclaw, Poland}
\email{jan.schneider@pwr.edu.pl}

\author[R. Urban]{Roman Urban}
\address{Institute of Mathematics\\
Wroclaw University\\
Plac Grunwaldzki 2/4\\
50-384 Wroclaw, Poland}
\email{urban@math.uni.wroc.pl}
\subjclass[2000]{}

\begin{document}

\begin{abstract}
In this note we prove a fuzzy analogue of Donsker's classical invariance principle.  - We consider a fuzzy random walk ${S^*_n}=X^*_1+\cdots+X^*_n,$ where $\{X^*_i\}_1^{\infty}$ is a sequence of mutually independent and identically distributed $d$-dimensional fuzzy random variables whose $\alpha$-cuts are assumed to be compact and convex. \par
Our reasoning and technique are based on the well known conjugacy correspondence between convex sets and support functions, which allows for the association of an appropriately normalized and interpolated time-continuous fuzzy random process with a real valued random process in the space of support functions. \par
We show that each member of the associated family of dual sequences tends in distribution to a standard Brownian motion.
\end{abstract}

\maketitle

\section*{Introduction}

This paper extends Monroe D. Donsker's classical invariance principle~\cite{Don52} by incorporating fuzzy random processes in the theory. That is, we are looking at stochastic processes consisting of a collection of fuzzy random variables in the sense of Puri and Ralescu~\cite{PR91}, indexed by a continuous parameter, defined on some probability space, and taking values in a space of fuzzy vectors endowed with an appropriate metric and associated Borel sets.

\smallskip
We show that the fuzzy random walk\footnote{A superscript ${}^*$ in our notation means that we have a fuzzy object.}, ${S^*_n}={X^*_1}+\cdots+X^*_n$  may be appropriately normalized and interpolated to a time-continuous stochastic process and then associated with a time continuous process of with values in $\R$ which tends in distribution to a standard Brownian motion in the space of support functions.\par

\smallskip
Probability theory and stochastic processes in the context of fuzzy vectors has attracted considerable attention in recent years (just to name a few authors):\par\vspace{1mm}
Fuzzy random variables were first introduced by Kwakernaak in 1978~\cite{Kwa78}. The concept of a fuzzy random variable was further developed in different ways by several authors. In our work we draw on Puri and Ralescu~(1986)~\cite{PR86}.
Klement, Puri and Ralescu~\cite{Ketal}, Kr\"{a}tschmer~\cite{K} and Wu~\cite{Wu} worked on central limit theorems, being the all-important tool of developed statistical analysis of fuzzy (just as in the case of real) data.

\smallskip
Fuzzy stochastic processes and their applications in the natural and social sciences have been investigated by a number of authors, including Puri and Ralescu~\cite{PR91}), and Brownian motion concretely and applications was investigated in \cite{Fen00}, \cite{LONH01}, \cite{LG07}, \cite{Mic11}, \cite{Bon12}, \cite{Fei13}, \cite{Mal14}, to name a few.\par

\smallskip
Our reasoning and proof utilizes the specific duality between characterizing functions of fuzzy vectors, and their corresponding support functions.

\smallskip
This note consists of two sections:

\smallskip
In section~1 the requisite conceptual groundwork of fuzzy and convex analysis is set forth: fuzzy vectors, the interrelation between characterizing and support functions, fuzzy random variables and processes.\vspace{0.3mm}\par

In section~2 we state and prove the main theorem of this note.

\section{Fuzzy random variables}
In this section we introduce the basic concepts and give some definitions. \par{\vspace{1mm}

\subsection{Fuzzy vectors}
A {\em $d$-dimensional fuzzy vector $x^*$} is defined by and may be identified with its {\em characterizing function} (see e.g. \cite{V,V0}). In this paper we work with vectors whose characterizing functions satisfy
\begin{definition}\label{defxi}
The characterizing function $\xi_{x^*}$
of a $d$-dimensional fuzzy vector $x^*$ is a function $\xi_{x^*}:\R^d\to\R$ satisfying:
\begin{itemize}
\item[1)] $\xi_{x^*}:\R^d\to [0,1],$
\item[2)] $\supp{\xi_{x^*}}$ is bounded,
\item[3)] for every $\alpha\in(0,1]$ the so called $\alpha$-cut $C_\alpha(x^*)$ of the fuzzy vector $x^*$,
$C_\alpha(x^*)= \{x\in\R^d:\xi_{x^*}(x)\geq\alpha\}$
is a non-empty, compact, and convex set, as is $C_0(x^*):=\overline{\{x\in\R^d: \xi_{x^*}(x)>0\}}=\supp\xi_{x^*}.$
\end{itemize}
\end{definition}

In the following we denote the set of all $d$-dimensional fuzzy vectors satisfying Def. \ref{defxi} by $\mathcal F_{\text{cconv}}(\R^d)$.

\begin{note}
There are many more general definitions of fuzzy vectors. For example in \cite[Definition 2.3]{V} assumption 3) is replaced by 3$^\prime$) for every $\alpha\in(0,1]$ the $\alpha$-cut $C_\alpha(x^*)$ is non-empty, bounded, and a finite union of connected and closed bounded sets. (See also \cite{V0}). Grzegorzewski, \cite[Definition 2.1]{G}, instead of our 3) assumes 3$^\prime$) which demands that the characterizing  function of the fuzzy vector $x^*$ be {\em fuzzy convex}, i.e., for all $\lambda\in[0,1]$ and all $x,y\in\R^d,$ \begin{equation*}\label{wypuklosc}
\xi_{x^*}(\lambda x+(1-\lambda)y)\geq\min\left(\xi_{x^*}(x),\xi_{x^*}(y)\right).
\end{equation*}
\end{note}
\subsubsection{Fuzzy arithmetic}
We start with the Minkowski arithmetic performed on subsets of $\R^d.$
\begin{definition}
Let $A,B\subset\R^d$ and $\lambda\in\R$. Then
\begin{equation*}
\begin{aligned}
A+B:=&\{a+b:\,a\in A,\,B\in B\},\\
A\cdot B:=&\{ab:\,a\in A,\,b\in B\},\\
\lambda\cdot A:=&\{\lambda a:\,a\in A\}.
\end{aligned}
\end{equation*}
\end{definition}
The Minkowski arithmetic of sets immediately induces an arithmetic of $d$-di\-men\-sio\-nal vectors from $\mathcal F_{\text{cconv}}(\R^d)$ via their $\alpha$-cuts:
\begin{definition}
The sum $x^*\oplus y^*$ and multiplication $x^*\odot y^*$
of two fuzzy $d$-dimensional fuzzy vectors $x^*$ and $y^*$
are defined via $\alpha$-cuts as follows
\begin{equation*}
\begin{aligned}
C_\alpha(x^*\oplus y^*)=\; &C_\alpha(x^*)+C_\alpha(y^*),\\
C_\alpha(x^*\odot y^*)=\; &C_\alpha(x^*)\cdot C_\alpha(y^*).
\end{aligned}
\end{equation*}
Similarly, the multiplication of a fuzzy vector $x^*$ by a crisp real number $\lambda,$ $\lambda\odot x^*$ is defined by the equation
\begin{equation*}
C_\alpha(\lambda\odot x^*)=\lambda\cdot C_\alpha(x^*).
\end{equation*}
\end{definition}
\subsubsection{\textbf{Notation}}
Let $K(\R^d)$  denote the set of all non-empty, closed subsets of $\R^d.$ By $K_{\text{c}}(\R^d)$ ($K_{\text{conv}}(\R^d)$, resp.) we denote the non-empty space of all compact subsets of $\R^d$ (the non-empty space of all closed convex subsets of $\R^d$, resp.) and finally, $K_{\text{cconv}}(\R^d)$ is the space of all non-empty compact and convex subsets of $\R^d$.\vspace{1mm}\par\noindent
The space $\R^d$ is equipped with the classical $\ell^2$-norm $\|x\|_{\ell^2}=\left(x_1^2+\cdots+x_d^2\right)^{1\slash 2}$ and the inner product $\langle x,y\rangle =\sum_1^dx_iy_i.$

\smallskip
In order to avoid trivialities we assume throughout the paper that $d>1.$ In the case $d=1$ we would be dealing with fuzzy intervals (numbers), which require a different set of techniques.

\subsection{Support function}\label{support}
Let $C\subset K_{\text{cconv}}(\R^d).$ By  $\s^{d-1}$ we denote the unit sphere in $\R^d,$ i.e.,  $\s^{d-1}=\{x\in\R^d: \|x\|_{\ell^2}=1\}.$
\begin{definition} The {\em support function of a set} $C\subset K_{\text{cconv}}(\R^d)$ is the function  $s_C:\s^{d-1}\to\R$ defined by
\begin{equation*}
s_C(u)=\sup_{a\in C} \langle u,a\rangle,\qquad u\in \s^{d-1}.
\end{equation*}
\end{definition}
\begin{definition}\label{sfff}
Let $x^*$ be a $d$-dimensional fuzzy vector. By 3) of Definition~\ref{defxi} its $\alpha$-cuts $C_\alpha(x^*)$ belong to $K_{\text{cconv}}(\R^d)$ so we can define the {\em support function of a fuzzy vector} $x^*$ as follows
\begin{equation}\label{sf}
s_{x^*}(\alpha,u)=\sup_{a\in C_\alpha(x^*)}\langle u,a\rangle,\qquad\alpha\in(0,1]\text{ and }u\in \s^{d-1}.
\end{equation}
\end{definition}
The support function $s_{x^*}(\cdot,\cdot)$ of $x^*\in\mathcal F_{\text{cconv}}(\R^d)$ has the following properties (\cite{DKbook,Korner}):
\begin{itemize}
\item [(i)] For every $\alpha\in(0,1]$, $s_{x^*}(\alpha,\cdot):\s^{d-1}\to\R$ is a  continuous function.
\item [(ii)] The support function is positive homogeneous with respect to the $u$ variable, i.e., for all real $\lambda\geq0$ and all $\alpha\in(0,1]$, $s_{x^*}(\alpha,\lambda u)=\lambda     s_{x^*}(\alpha,u).$
\item [(iii)] For every $\alpha\in(0,1]$, $s_{x^*}(\alpha,\cdot)$ is sub-additive, i.e., for all $u,v\in \s^{d-1},$ $s_{x^*}(\alpha,u+v)\leq  s_{x^*}(\alpha,u)+s_{x^*}(\alpha,v).$
\item [(iv)] For every $u\in \s^{d-1}$ the function $s_{x^*}(\cdot,u):(0,1]\to\R$ is left continuous and non-increasing, i.e., for all $0<\alpha\leq\beta\leq 1,$ $s_{x^*}(\alpha,u)\geq s_{x^*}(\beta,u).$
\end{itemize}

The proof of the following lemma is straightforward:
\begin{lemma}\label{addsup}
For every $x^*, y^*\in\mathcal F_{\text{cconv}}(\R^d)$ and for every $\lambda\in\R,$ we have that
\begin{equation}\label{additive}
s_{x^*\oplus y^*}(\alpha,u)=s_{x^*}(\alpha,u)+s_{y^*}(\alpha,u)
\end{equation}
and
\begin{equation}\label{pseudolin}
s_{\lambda\odot x^*}(\alpha,u)=|\lambda| s_{\sign{(\lambda)} \odot x^*}(\alpha,u).
\end{equation}
\end{lemma}\vspace{1mm}
It turns out that integrable functions with properties (i)-(iv) completely characterize the elements of $\mathcal  F_{\text{cconv}}(\R^d).$ Specifically, we have the following result:
\begin{theorem}\label{inverse}
For every Lebesgue integrable function $f\in L^1\left((0,1]\times \s^{d-1}\right)$ with properties (i)-(iv) there exists exactly one $d$-dimensional fuzzy vector $x^*\in\mathcal F_{\text{cconv}}(\R^d)$ such that for every $\alpha\in(0,1]$,
\begin{equation*}
C_\alpha(x^*)=\{x\in\R^d:\,\langle u,x\rangle\leq f(\alpha,u)\text{ for every }u\in \s^{d-1}\}
\end{equation*}
and $s_{x^*}=f.$
\end{theorem}
For a proof see \cite{KN}.
\subsection{Metrics on the space of fuzzy vectors}
The {\em Hausdorff metric} $d_H$ on the space  $K_{\text{cconv}}(\R^d)$ is given by the following  formula:
\begin{equation*}
d_H(A,B)=\max\left\{\sup_{a\in A}\inf_{b\in B}\|a-b\|_{\ell^2},\;\sup_{b\in B}\inf_{a\in A}\|a-b\|_{\ell^2}\right\},
\end{equation*}
where $A,B\in K_{\text{cconv}}(\R^d)$.

\smallskip
Based on the Hausdorff metric $d_H$ one can construct $L^p$-metrics, $1\leq p\leq\infty,$ on $\mathcal F_{\text{cconv}}(\R^d)$ via $\alpha$-cuts:\par\vspace{1mm}
For $1\leq p<\infty:$
\begin{equation}\label{dp}
d_p(x^*,y^*)=
\left(\int_0^1d_H\bigl(C_\alpha(x^*),C_\alpha(y^*)\bigr)^p
d\alpha\right)^{1\slash p},
\end{equation}
\indent and for $p=\infty$:
\begin{equation}\label{d8}
d_\infty(x^*,y^*)=\sup_{\alpha\in [0,1]}d_H\bigl(C_\alpha(x^*),C_\alpha(y^*)\bigr),\;x^*,
y^*\in\mathcal F_{\text{cconv}}(\R^d).
\end{equation}

The space $\mathcal F_{\text{cconv}}(\R^d)$ equipped with $d_\infty$ is a complete metric space (but not separable). But for $1\leq p<\infty$ the metric space $(\mathcal F_{\text{cconv}}(\R^d),\,d_p)$ is both complete and separable (see \cite{DK1,DK2}).

\smallskip
\begin{remark}
The Hausdorff distance between two vectors $x^*,y^*\in\mathcal F_{\text{cconv}}(\R^d)$ may be expressed by means of their support functions. Namely,
\begin{equation*}
d_H(x^*,y^*)=\sup_{u\in \s^{d-1}}|s_{x^*}(u)-s_{y^*}(u)|,\qquad x^*,y^*\in\mathcal F_{\text{cconv}}(\R^d).
\end{equation*}
(See \cite[p. 243]{DK1}.)
\end{remark}

\subsection{The embedding theorem for $\mathcal F_{\text{cconv}}(\R^d)$}
For $1\leq p\leq\infty,$ there exists an isometrically isomorphic embedding of $\left(\mathcal F_{\text{cconv}}(\R^d),d_p\right)$ into the Banach space $L^p\bigl((0,1]\times\s^{d-1},\,dxd\lambda\bigr),$ where $\lambda$ is the (normalized) Lebesgue measure on $\s^{d-1},$ onto a positive cone $\mathcal H\subset L^p\left((0,1]\times\s^{d-1}\right).$ The space $L^p\left((0,1]\times\s^{d-1}\right)$ is equipped with a standard $L^p$-norm (and metric). Namely,
for $1\leq p<\infty$ the norm is given by:
\begin{equation*}
\|f\|_p=\left(\int_0^1\int_{\s^{d-1}}|f(x,u)|^pdx d\lambda(u)\right)^{1\slash p},
\end{equation*}
and the corresponding distance function is
\begin{equation}\label{rhop}
\rho_p(f,g)=\|f-g\|_p \,.
\end{equation}
For $p=\infty,$
$$\|f\|_\infty=\sup_{\alpha\in(0,1]}\sup_{u\in \s^{d-1}}|
f(\alpha,u)|\,,$$
and the corresponding metric is
\begin{equation*}
\rho_\infty(f,g)=\|f-g\|_\infty.
\end{equation*}

The embedding is defined as follows:
\begin{theorem}\label{embthm}
Let
\begin{equation*}
j:\mathcal F_{\text{cconv}}(\R^d)\to L^p\left((0,1]\times \s^{d-1}\right)
\end{equation*}
be defined by
\begin{equation}\label{j}
j(x^*)\mapsto s_{x^*}(\cdot,\cdot).
\end{equation}

\smallskip\noindent
The mapping $j$ is {\em positive linear}\footnote{Positive linearity of $j$ \eqref{al} follows from Lemma~\ref{addsup}.},i.e., for all non-negative real numbers $\lambda,\mu$ we have
\begin{equation}\label{al}
j(\lambda\odot x^*+\mu\odot y^*)=\lambda j(x^*)+\mu j(y^*),\text{ for }\lambda,\mu\geq 0.
\end{equation}
The $j$-map is one-to-one and onto its image $j\left(\mathcal F_{\text{cconv}}(\R^d)\right)$ which is a closed and convex cone in $L^p\left((0,1]\times \s^{d-1}\right).$
Moreover, for all $1\leq p\leq\infty,$ the mapping $j$ is an isometry,
\begin{equation}\label{izom}
d_p(x^*,y^*)=\rho_p(j(x^*),j(y^*)).
\end{equation}
\end{theorem}
\begin{proof}
For the proof see \cite[Theorem 4.3]{K}, \cite[p. 158]{V} and the literature cited therein.
\end{proof}
\begin{note}
 By theorems~\ref{inverse} and \ref{embthm} we may identify the space of $d$-dimensional fuzzy vectors with a certain cone $\mathcal H$ of real functions defined on the product space $(0,1]\times\s^{d-1}.$
\end{note}
\subsection{Normal random variables and Gaussian processes with values in the space of  $d$-dimensional fuzzy vectors. Independence}

\begin{definition}
By a fuzzy random variable in our setting we understand a measurable function $X^*:(\Omega,\mathcal F,\p)\to (\mathcal F_{\text{cconv}}(\R^d),\mathcal B) =:S.$\par
\noindent Here $(\Omega,\mathcal F,\p)$ is some probability space and $S$ is the space of fuzzy vectors, equipped with a metric $d_p$, for some $p\in[1,\infty].$ The associated Borel $\sigma$-field $\mathcal B$ is generated by open balls which are open in the chosen metric $d_p$.

As usual, $X^*$ is termed an $(\mathcal F - \mathcal B)$-measurable function iff for every $B\in\mathcal B$ the inverse image ${X^*}^{-1}(B):=\{\omega\in\Omega:X^*(\omega)\in B\}$ belongs to the $\sigma$-field $\mathcal F.$
\end{definition}

\noindent
The notion of independence of fuzzy random variables transfers from the classical case verbatim:
\begin{definition}
Fuzzy random variables $X^*$ and $Y^*$ are independent if and only if $\p(X^*\in B_1\text{ and }Y^*\in B_2)=\p(X^*\in B_1)\cdot\p(Y^*\in B_2),$ for all $B_1,B_2$ belonging to the pertinent $\sigma$-field $\mathcal{B}.$
\end{definition}

The following Lemma is key to our investigations:
\begin{lemma}\label{i}
Let $\{X^*_n\}_{n=1}^\infty$ be a sequence of mutually independent fuzzy random variables with values in the metric space $\left(\mathcal F_{\text{cconv}}(\R^d),d_p\right)$ (with the $\sigma$-field $\mathcal B$ generated by open balls in the metric $d_p).$ \par
Then the sequence of support functions
$$\left\{s_{X^*_n}(\cdot,\cdot)\}_{n=1}^\infty = \{j({X^*_n})(\cdot,\cdot)\right\}_{n=1}^\infty$$ embedded in the metric space $S=\left(L^p((0,1]\times\s^{d-1},\rho_p\right)$ (with the Borel $\sigma$-field $\mathcal S$ generated by the metric $\rho_p$) and evaluated at arbitrary point $(\alpha,u)\in (0,1]\times\s^{d-1}$, i.e.
\begin{equation*}
s_{X^*_n}(\alpha,u)(\omega)=s_{X^*_n(\omega)}(\alpha,u)
\end{equation*}

\noindent consists of mutually independent real random variables.
\end{lemma}
\begin{proof}
Take $m\not=n.$ Let $f_{(\alpha,u)}:S\to\R,$ $f_{(\alpha,u)}(\varphi)=\varphi(\alpha,u).$ If $\varphi=s_{X^*_j}$ for some $j,$ then by properties (ii) and (iv) in section~\ref{support} of the support function, it follows that $f_{(\alpha,u)}$ is well defined. Therefore, for two Borel sets $B_1,B_2\subset\R$ from the Borel $\sigma$-field $\mathcal R,$
\begin{multline*}
\p\left(s_{X^*_m}(\alpha,u)\in B_1\text{ and }s_{X^*_n}(\alpha,u)\in B_2\right)\\=\p\left(f_{(\alpha,u)}\circ j(X^*_m)\in B_1\text{ and }f_{(\alpha,u)}\circ j(X^*_n)\in B_2\right)\\=\p\left(X^*_m\in (f_{(\alpha,u)}\circ j)^{-1}(B_1)\text{ and }X^*_n\in (f_{(\alpha,u)}\circ j)^{-1}(B_2)\right).
\end{multline*}
Since $j$ is $(\mathcal B - \mathcal S)$-measurable and $f_{(\alpha,u)}$ is $(\mathcal S - \mathcal R)$-measurable it follows that $f_{(\alpha,u)}\circ j$ is $(\mathcal B - \mathcal R)$-measurable. Hence, for $i=1,2$ the respective inverse images $(f_{(\alpha,u)}\circ j)^{-1}(B_i)\in\mathcal B.$ Thus by independence of $X^*_m,X^*_n$ we get
\begin{multline*}
\p\left(s_{X^*_m}(\alpha,u)\in B_1\text{ and }s_{X^*_n}(\alpha,u)\in B_2\right)\\=\p\left(X^*_m\in (f_{(\alpha,u)}\circ j)^{-1}(B_1)\right)\p\left(X^*_n\in (f_{(\alpha,u)}\circ j)^{-1}(B_2)\right).
\end{multline*}
Thus,
\begin{multline*}
\p\left(s_{X^*_m}(\alpha,u)\in B_1\text{ and }s_{X^*_n}(\alpha,u)\in B_2\right)\\=\p\left( (f_{(\alpha,u)}\circ j)(X^*_m)\in B_1\right)\p\left(X^*_n\in (f_{(\alpha,u)}\circ j)(X^*_n)\in B_2\right)\\
\p\left(s_{X^*_m}(\alpha,u)\in B_1\right)\p\left(s_{X^*_n}(\alpha,u)\in B_2\right).
\end{multline*}
\end{proof}
To recapitulate: by the embedding theorem (Theorem~\ref{embthm}) a fuzzy random variable $X^*(\cdot)$ with values in $\mathcal F_{\text{cconv}}(\R^d)$ can be identified with a collection of real-valued random variables $s_{X^*}(\alpha,u)(\cdot)$, $(\alpha,u) \in$ (the Banach space) $L^p\left((0,1]\times \s^{d-1}\right)$. The same identification serves to define the notion of normality of a fuzzy random variable.
\subsection{Expectation and variance of fuzzy random variables}
There are many definitions of expectation and variance of fuzzy random variables with values in $\mathcal F_{\text{cconv}}(\R^d).$ (See \cite{V} and references therein.)
Below we give the definition and some properties of the Bochner expectation:

\subsubsection{\textbf{Bochner expectation}}
\begin{definition}\label{dbochner}
Let ${X^*}$ be a $d$-dimensional fuzzy variable. The {\em Bochner expectation} $\e^B {X^*}$ of ${X^*}$ is defined via its associated support function,
\begin{equation}\label{be}
s_{\,\e^B {X^*}}(\alpha,u)=\e\,s_{X^*}(\alpha,u),\text{for all }(\alpha,u).
\end{equation}
\end{definition}
The Bochner expectation is well defined since for every random variable ${X^*}$ with values in $\mathcal F_{\text{cconv}}(\R^d)$ and for $\alpha\in(0,1],$ and $u\in\s^{d-1}$ fixed, the support function $s_{X^*}(\alpha,u)$ is a real random variable. Hence the expectation on the right is the classical expectation of a real random variable.\vspace{1mm}

The Bochner expectation is a fuzzy vector the characterizing function of which can be reconstructed from equality \eqref{be} by  Theorem~\ref{inverse}.
\begin{lemma}\label{econv0}
Let ${X^*}$ be a random variable with values in $\mathcal F_{\text{cconv}}(\R^d).$ Then the Bochner expectation also belongs to $\mathcal F_{\text{cconv}}(\R^d).$
\end{lemma}
\begin{proof}
See \cite{Vit,Korner-b}.
\end{proof}
\par\noindent
By Lemma~\ref{addsup} and linearity of the classical expectation $\e$ we get
\begin{lemma}
For all $d$-dimensional fuzzy random variables ${X^*}$ and ${Y^*}$
\begin{multline}\label{additive1}
s_{\e^B({X^*}\oplus {Y^*})}(\alpha,u)=\e s_{{X^*}\oplus {Y^*}}(\alpha,u)\\=\e \left(s_{X^*}(\alpha,u)+s_{Y^*}(\alpha,u)\right)=\e s_{{X^*}}(\alpha,u)+\e s_{{Y^*}}(\alpha,u)
\end{multline}
and, for every $\lambda\in\R,$
\begin{equation}\label{pseudolin1}
s_{\e^B(\lambda\odot {X^*})}(\alpha,u)=\e s_{\lambda\odot {X^*}}(\alpha,u)=|\lambda|\e s_{\sign{(\lambda)} \odot {X^*}}(\alpha,u).
\end{equation}
\end{lemma}$ $\par

As a corollary from \eqref{additive1} and \eqref{pseudolin} we get
\begin{lemma}
The Bochner expectation is positive linear, i.e., for all $d$-di\-men\-sio\-nal fuzzy random variables ${X^*}$ and ${Y^*}$ with values in $\mathcal F_{\text{cconv}}(\R^d)$ and for all non-negative real numbers $\lambda$ and $\mu,$
\begin{equation*}
\e^B\left(\lambda\odot {X^*}\oplus\mu\odot {Y^*}\right)=\lambda\odot\e^B{X^*}\oplus\mu\odot\e^B{Y^*},
\text{ for }\lambda,\mu\geq 0.
\end{equation*}
\end{lemma}
\begin{remark}
Although the Bochner expectation is only positive linear, it has sufficiently good properties for our purposes.\end{remark}
\subsubsection{\textbf{Fr\'echet variance of fuzzy vectors}}
In 1948 M. Fr\'echet gave a definition of the expectation and variance of a fuzzy random variable \cite{Frechet}. Here we will be using a modified definition of Fr\'echet's variance, replacing in \eqref{variance} Fr\'echet's expectation $\e^F$ in the original definition with the Bochner expectation $\e^B$.
\begin{note}
This modification arises quite naturally with the observation that, in the $L^2$-metric space
\begin{equation*}
\F:=\left(\mathcal F_{\text{cconv}}(\R^d),d_2\right)\overset{iso}{=}\left(j\left(\mathcal F_{\text{cconv}}(\R^d)\right),\rho_2\right)
\end{equation*}
the Bochner and Fr\'echet (and for that matter also the Aumann) expectations are equal:
\begin{equation*}
\e^B=\e^F.
\end{equation*}
For the proof see \cite{Korner-b}.
\end{note}
\begin{definition}
The Frech\'et variance $\Var^F {X^*}$ of a $d$-dimensional fuzzy random variable ${X^*}$ is defined by integrating its support function:
\begin{equation}\label{variance}
\begin{split}
\Var^F {X^*}=&\e d_2^2({X^*},\e^B{X^*})=\e\rho_2^2(s_{X^*},s_{\e^B{X^*}})\\
=&\e\int_0^1\int_{\s^{d-1}}|s_{X^*}(\alpha,u)-s_{\e^B {X^*}}(\alpha,u)|^2d\lambda(u)d\alpha\\
=&\int_0^1\int_{\s^{d-1}}\e |s_{X^*}(\alpha,u)
-\e^Bs_{{X^*}}(\alpha,u)|^2d\lambda(u)d\alpha\\
=&\|\Var s_{X^*}\|_{L^1((0,1]\times\s^{d-1})}.
\end{split}
\end{equation}
\end{definition}
\vspace{1mm}

\begin{lemma}\label{expvar}
Let ${X^*}$ be a $d$-dimensional random variable with values in $\mathcal{F}_{\textrm{cconv}}(\R^d).$
\begin{itemize}
\item[i)] If $\e^B {X^*}$ exists then $\e s_{{X^*}}(\alpha,u)$ exists for every $\alpha\in(0,1]$ and $u\in \s^{d-1}.$
\item[ii)] If $\Var^B {X^*}$ (variance using the Bochner expectation $\e^B$) is finite then (the classical variance) $\Var s_{X^*}(\alpha,u)$ is finite for every $\alpha\in(0,1]$ and $u\in \s^{d-1}.$
\end{itemize}
\end{lemma}
\begin{proof}$ $\par
\begin{itemize}
\item[i)] By definition \eqref{dbochner} of the Bochner expectation
\begin{equation}\label{1p}
s_{\e^B{X^*}}(\alpha,u)=\e s_{X^*}(\alpha,u).
\end{equation}
So, if $\e^B{X^*}$ exists then, for every $(\alpha,u)\in(0,1]\times\s^{d-1},$ the classical expectation $\e s_{X^*}(\alpha,u)$ exists.
\item[ii)] We apply \eqref{1p} and use the regularity properties of support functions given just after Definition \eqref{sfff}.
\end{itemize}
\end{proof}

\section{Donsker's type theorem for fuzzy random variables}\label{approx}
Let $X_{n}^*$, $n=1,2,\ldots$ be a sequence of mutually independent identically distributed fuzzy random variables defined on some probability space $(\Omega,\mathcal B,\p)$ and taking values in the space $\mathcal F_{\text{cconv}}(\R^d)$ of $d$-dimensional fuzzy vectors equipped with a metric $d_p$ for some $p\in[1,\infty)$ as defined in \eqref{dp}. (Recall that $(\mathcal F_{\text{cconv}}(\R^d), d_p),$ for $1\leq p<\infty$ is a complete and separable metric space whereas if $p=\infty$ the metric space $(\mathcal  F_{\text{cconv}}(\R^d), d_\infty)$ is only complete.)
\vspace{1mm}\par\noindent
We also assume that
\begin{equation}\label{zal}
\e^BX_j^*=m^*\text{ exists and }\Var^F X_j^*=\sigma^2<+\infty,
\end{equation}
where the expectation operator $\e^B$ (the variance $\Var^F,$ resp.) is defined in \eqref{be} (in \eqref{variance}, resp.).
\vspace{1mm}\par\noindent
Let $(\alpha,u)$ be fixed but otherwise arbitrary. We denote
\begin{equation*}
\tilde\sigma^2=\Var s_{X_j^*}(\alpha,u),
\end{equation*}
for all $j\geq1,$ where $\Var$ is the classical variance of real random variable.

\smallskip
In the following we generalize the construction of a standard (i.e., taking values in $\R$) Brownian motion, (see e.g. \cite[p. 220]{P}, \cite[p. 66]{KS}, \cite[p. 518]{RY}).\vspace{1mm}\par\noindent
Let
\begin{equation*}
X_0^*=0,\, S_{0}^*=0,\, S_{k}^*=\bigoplus_{i=1}^kX_{i}^*.
\end{equation*}
The partial sums $S_k^*$ can be thought of as a natural generalization of the standard random walk in $\R.$
\vspace{1mm}\par\noindent
With the sequence $\{S_k^*\}_{k=0}^\infty$ we now associate, by means of linear interpolation, a continuous-time process
\begin{equation}\label{lnt}
L_t^*=S_{\lfloor t\rfloor}
\oplus(t-\lfloor t\rfloor)\odot X_{\lfloor t\rfloor+1}^*\text{ for }t\geq 0,
\end{equation}
where $\lfloor t\rfloor$ is the {\em floor function}, i.e., the greatest integer less than or equal to $t.$

\noindent We now normalize $L_t^*$ appropriately to obtain another time-continuous process:
\begin{equation}\label{mnt}
M_{t,n}^*=L_{nt}^*\slash(\tilde\sigma\sqrt n).
\end{equation}
The processes  $L_t^*$ and $M_t^*$ are elements of the function space $$C\left([0,\infty),(\mathcal  F_{\text{cconv}}(\R^d),d_p)\right),$$ i.e., the space of continuous functions $f:[0,\infty)\to\mathcal F_{\text{cconv}}(\R^d).$ However, we do not want to work with this function space.
Instead, using the $j$-map defined in \eqref{j}, we embed $M_t^*$ into the space $$L^p\left((0,1]\times \s^{d-1}\right)$$ which is more suitable for our considerations.
\par
At first we compute the support function $s_{M_{t,n}^*}$ of the process $M_{t,n}^*$ defined by \eqref{lnt} and \eqref{mnt}. By properties (i)-(iv) of Definition \ref{defxi} of the support function and Lemma~\ref{addsup} we have
\begin{equation}\label{smtn}
\begin{split}
s_{M_{t,n}^*}\,(\alpha,u)=&
\frac{1}{\tilde\sigma\sqrt n}s_{L_{nt}^*}(\alpha,u)\\
=&\frac{1}{\tilde\sigma\sqrt n}\left(s_{S_{\lfloor nt\rfloor}}(\alpha,u)
+(nt-\lfloor nt\rfloor)s_{X_{\lfloor nt\rfloor+1}^*}(\alpha,u)\right).
\end{split}
\end{equation}
\begin{remark}\label{rrv}
The mapping $t\mapsto s_{M_{t,n}^*(\omega)}(\cdot,\cdot)$ (for $\omega\in\Omega$ fixed) is an element of the space
$\mathcal C=C\left([0,\infty),L^p((0,1]\times\s^{d-1})\right).$
Recall that on the other hand, for every random variable $X^*$ with values in $\mathcal F_{\text{cconv}}(\R^d)$ and for $\alpha$ and $u$ fixed, the support function $s_{X^*(\cdot)}(\alpha,u)$ is a real random variable.
\end{remark}
\par
Now we are ready to prove the fuzzy analogue of the classical Donsker invariance theorem.
In the next section we prove that, when appropriately normalized (similarly to the classical case),
the random walk $\{S_k^*\}$ approximates Brownian motion. Since Brownian motion is time-continuous process we also had to interpolate $\{S_k^*\}.$

To ensure convergence of $s_{M_{t,n}^*}$ we must appropriately subtract expectations. Otherwise the process would have a drift going to infinity.
Therefore we consider a modified process. For better readability we drop in our writings below the argument $(\alpha,u)$ of the support functions, and the $B$ superscript of the Bochner expectation. (Recall that for all $j,$ we write $\e^BX_j^*=m^*.$)
\begin{multline}\label{mtn1}
\tilde s_{M_{t,n}^*}=s_{M_{t,n}^*}-\frac{1}{\tilde\sigma\sqrt{n}}
(\lfloor nt\rfloor+\left(nt-\lfloor nt\rfloor)\right)j(m^*)\\
=\frac{1}{\tilde\sigma\sqrt n}\left(s_{S_{\lfloor nt\rfloor}^*}-\lfloor nt\rfloor
s_{m^*}+(nt-\lfloor nt\rfloor)s_{X_{\lfloor nt\rfloor}^*+1}-(nt-\lfloor nt\rfloor)s_{m^*}\right)
\end{multline}
\begin{theorem}\label{main1}
Let the setting be as in section~1. Let $\tilde s_{M_{t,n}^*}$ be defined as in \eqref{mtn1} and let $b_t$ denote the classical Brownian motion in $\R.$ We assume that \eqref{zal} is satisfied. Then for every $k\in\N$ and for every finite sequence of times $0\leq t_1<t_2<\cdots<t_k,$ and for every $\alpha\in(0,1],$ and $u\in\s^{d-1},$ we have that
\begin{equation*}
\left(\tilde s_{M_{t_1,n}^*}(\alpha,u),
\ldots,\tilde s_{M_{t_k,n}^*}(\alpha,u)\right)
\to\left(b_{t_1},\ldots,b_{t_k}\right)
\end{equation*}
in distribution as $n\to\infty.$
\end{theorem}

\begin{note}
For the classical Donsker invariance principle there are a number of different proofs available in the literature (see e.g. \cite{P,KS,RY}. All of them are essentially the same, the differences are in techniques preferred by the authors.
We found the approach given in \cite{KS} to be the most suitable for our purpose of generalizing the Donsker theorem to a fuzzy context. Hence we follow the main steps of the proof given in \cite[p. 67]{KS} adjusting them appropriately to the fuzzy setting.
\end{note}
Before proceeding to the proof of Theorem~\ref{main1} we need two simple lemmas about convergence of classical random variables in metric spaces (see e. g. \cite{P}).
\subsubsection{Convergence of probability measures}
\begin{lemma}\label{pl}
Let $X_n,$ $Y_n,$ $n=1,2,\ldots,$ and $X$ be random variables taking values in a separable metric space $(S,d).$ Suppose that for every fixed $n\geq1$ the random variables $X_n$ and $Y_n$ are defined on the same probability space.\vspace{0.5mm}

\noindent If, as $n$ tends to infinity,
\begin{equation*}
X_n\longrightarrow X\text{ in distribution and }d(X_n,Y_n)\longrightarrow 0\text{ in probability,}
\end{equation*}
then
\begin{equation*}
Y_n\longrightarrow X\text{ in distribution.}
\end{equation*}
\end{lemma}
\begin{lemma}\label{pl2}
Let $X_1,X_2\ldots,$ be a sequence of random variables with values in a metric space $(S_1,d_1).$ Assume that as $n$ tends to infinity then
\begin{equation*}
X_n\longrightarrow X\text{ in distribution.}
\end{equation*}
Now consider a second metric space $(S_2,d_2).$ Let $f:S_1\to S_2$ be a continuous function. Then, as $n\to\infty,$ the sequence
\begin{equation*}
f(X_n)\longrightarrow f(X)\text{ in distribution.}
\end{equation*}
\end{lemma}
$ $\par

\smallskip
Now we have all that is needed for the proof of Theorem~\ref{main1}.
Again for better readability we write out the proof for $k=2$.
\begin{proof}[\textbf{Proof of Theorem~\ref{main1}}]$ $\vspace{1mm}\par\noindent

By \eqref{al} - the positive linearity of $j$ - equation~\eqref{smtn}, and property \eqref{1p} we have
\begin{equation*}
\left|\tilde s_{M_{t,n}^*}-\frac{1}{\tilde\sigma\sqrt{n}}
\left(s_{S_{\lfloor  nt\rfloor}^*}-\lfloor nt\rfloor s_{m^*}\right)\right|
\leq\frac{1}{\tilde\sigma\sqrt n}\left|s_{X_{\lfloor nt\rfloor+1}^*}-s_{m^*}\right|.
\end{equation*}
Hence, for every $\varepsilon>0,$ we get as a consequence of the previous in\-e\-qua\-li\-ty and Lemma~\ref{expvar} i) that
\begin{multline}\label{1}
\p\left(\left|\tilde s_{M_{t,n}^*}-\frac{1}{\tilde\sigma\sqrt{n}}\left(
s_{S_{\lfloor nt\rfloor}^*}-\lfloor nt\rfloor s_{m^*}\right)\right|>\varepsilon\right)\\
\leq\p\left(\frac{1}{\tilde\sigma\sqrt n}\left|s_{X_{\lfloor nt\rfloor+1}^*}-\e s_{X_{\lfloor nt\rfloor+1}^*}\right|>\varepsilon\right)
\end{multline}
and by Chebyshev's inequality applied to the right hand side of \eqref{1} we get
\begin{equation*}
\p\left(\left|\tilde s_{M_{t,n}^*}-\frac{1}{\tilde\sigma\sqrt{n}}
\left(
s_{S_{\lfloor nt\rfloor}^*}-\lfloor nt\rfloor s_{m^*}\right)\right|>\varepsilon\right)
\leq\frac{1}{\varepsilon^2n}
\longrightarrow 0\text{ as }n\longrightarrow\infty\,.
\end{equation*}\vspace{1mm}

\noindent Hence, for $t_1<t_2,$
\begin{multline*}
\left\|(\tilde s_{M_{t_1,n}^*},\tilde s_{M_{t_2,n}^*})-\frac{1}{\tilde\sigma\sqrt n}\left(s_{S_{\lfloor nt_1\rfloor}^*}-\lfloor nt_1\rfloor s_{m^*},s_{S_{\lfloor nt_2\rfloor}^*}-\lfloor nt_2\rfloor s_{m^*}\right)\right\|_{\ell^2(\R^2)}\\\longrightarrow 0\text{ in probability.}
\end{multline*}
By Lemma~\ref{pl} it is enough to show that
\begin{multline*}
\frac{1}{\tilde\sigma\sqrt n}\left(s_{S_{\lfloor nt_1\rfloor}^*}-\lfloor nt_1\rfloor s_{m^*},s_{S_{\lfloor nt_2\rfloor}^*}-\lfloor nt_2\rfloor s_{m^*}\right)\\
\longrightarrow\left(b_{t_1},b_{t_2}\right)\text{ in distribution.}
\end{multline*}
Equivalently, by Lemma~\ref{pl2}, taking $f:\R^2\to\R^2$ equal to $f(x_1,x_2)=(x_1,x_2-x_1)$ it is enough to show that
\begin{multline}\label{ft}
\frac{1}{\tilde\sigma\sqrt n}\left(\sum_{i=0}^{\lfloor nt_1\rfloor}\left(s_{X_i^*}-\e s_{X_i^*}\right),\sum_{\lfloor nt_1\rfloor+1}^{\lfloor nt_2\rfloor}\left(s_{X_i^*}-\e s_{X_i^*}\right)\right)\\
\longrightarrow\left(b_{t_1},b_{t_2}-b_{t_1}\right)
\text{ in distribution.}
\end{multline}

\vspace{1mm}\par\noindent
To simplify notation we introduce the following shorthand symbols:
\begin{equation}\label{vert}
\begin{split}
s_{\tilde S^*_{0,{\lfloor nt_1\rfloor}}}=&\frac{1}{\tilde\sigma\sqrt n}\sum_{i=0}^{\lfloor nt_1\rfloor}\left(s_{X_i^*}-\e s_{X_i^*}\right),\\
s_{\tilde S^*_{{\lfloor nt_1\rfloor}+1,{\lfloor nt_2\rfloor}}}=&\frac{1}{\tilde\sigma\sqrt n}\sum_{i={\lfloor nt_1\rfloor}+1}^{\lfloor nt_2\rfloor}\left(s_{X_i^*}-\e s_{X_i^*}\right).
\end{split}
\end{equation}
Now we compute the Fourier transform $F$ of the vector
\begin{equation*}
\left(s_{\tilde S^*_{0,\lfloor nt_1\rfloor}},s_{\tilde S^*_{\lfloor nt_1\rfloor+1,\lfloor nt_2\rfloor}}\right)
\end{equation*}
defined in \eqref{vert}. By independence (see Lemma~\ref{i}) we get
\begin{multline*}
F_{s_{\tilde S^*_{0,\lfloor nt_1\rfloor}},s_{\tilde S^*_{\lfloor nt_1\rfloor+1,\lfloor nt_2\rfloor}}}(u_1,u_2)\\=
\e\exp\left(\frac{iu_1}{\tilde\sigma\sqrt n}\sum_{k=0}^{\lfloor nt_1\rfloor}\left(s_{X_k^*}-\e s_{X_k^*}\right)+\frac{iu_2}{\tilde\sigma\sqrt n}\sum_{\lfloor nt_1\rfloor+1}^{\lfloor nt_2\rfloor}\left(s_{X_k^*}-\e s_{X_k^*}\right)\right)\\
=\e\exp\left(\frac{iu_1}{\tilde\sigma\sqrt n}\sum_{k=0}^{\lfloor nt_1\rfloor}\left(s_{X_k^*}-\e s_{X_k^*}\right)\right)\\\cdot
\e\exp\left(\frac{iu_2}{\tilde\sigma\sqrt n}\sum_{k=\lfloor nt_1\rfloor+1}^{\lfloor nt_2\rfloor}\left(s_{X_k^*}-\e s_{X_k^*}\right)\right)\\
=F_{s_{\tilde S^*_{0,\lfloor nt_1\rfloor}}}(u_1)F_{s_{\tilde S^*_{{\lfloor nt_1\rfloor}+1,\lfloor nt_2\rfloor}}}(u_2).
\end{multline*}
Thus,
\begin{equation}\label{Fourier}
\lim_{n\to\infty}F_{s_{\tilde S^*_{0,\lfloor nt_1\rfloor}},s_{\tilde S^*_{\lfloor nt_1\rfloor+1,\lfloor nt_2\rfloor}}}(u_1,u_2)=\lim_{n\to\infty}F_{s_{\tilde S^*_{0,\lfloor nt_1\rfloor}}}(u_1)\lim_{n\to\infty}F_{s_{\tilde S^*_{{\lfloor nt_1\rfloor}+1,\lfloor nt_2\rfloor}}}(u_2)
\end{equation}
provided that the limits on the right exist.
\vspace{1mm}\par\noindent
Therefore let us straightforwardly calculate those limits.
Take the first Fourier transform $F_{s_{\tilde S^*_{0,\lfloor nt_1\rfloor}}}$ on the right hand side of \eqref{Fourier}.
\vspace{1mm}\par\noindent
Clearly,
\begin{multline*}
\left|\frac{1}{\tilde\sigma\sqrt n}\sum_{i=1}^{\lfloor nt_1\rfloor}\left(s_{X_i^*}-\e s_{X_i}\right)-\frac{\sqrt{t_1}}{\tilde\sigma\sqrt{\lfloor nt_1\rfloor}}
\sum_{i=1}^{\lfloor nt_1\rfloor}\left(s_{X_i^*}-\e s_{X_i^*}\right)\right|\\
\longrightarrow 0\text{ in probability.}
\end{multline*}
In fact, rewriting and then applying Chebyshev's inequality we get
\begin{multline*}
\p\left(\left|\frac{1}{\tilde\sigma\sqrt n}\sum_{i=1}^{\lfloor nt_1\rfloor}\left(s_{X_{i}^*}-\e s_{X_{i}^*}\right)-\frac{\sqrt{t_1}}{\tilde\sigma\sqrt{\lfloor nt_1\rfloor}}
\sum_{i=1}^{\lfloor nt_1\rfloor}\left(s_{X_{i}^*}-\e s_{X_{i}^*}\right)\right|>\varepsilon\right)\\
=\p\left(\left|\left(\frac{1}{\tilde\sigma\sqrt n}-\frac{\sqrt{t_1}}{\tilde\sigma_n\sqrt{\lfloor nt_1\rfloor}}\right)
\sum_{i=1}^{\lfloor nt_1\rfloor}\left(s_{X_{i}^*}-\e s_{X_{i}^*}\right)\right|>\varepsilon\right)\\
\leq\frac{\tilde\sigma^2\left(\frac{1}{\tilde\sigma\sqrt n}-\frac{\sqrt{t_1}}{\tilde\sigma\sqrt{\lfloor nt_1\rfloor}}\right)^2}
{\varepsilon^2}\longrightarrow 0,
\end{multline*}
as $n$ goes to infinity.\vspace{1mm}

By the (classical) central limit theorem
\begin{equation*}
\frac{\sqrt{t_1}}{\tilde\sigma\sqrt{\lfloor nt_1\rfloor}}\sum_{i=1}^{\lfloor nt_1\rfloor}\left(s_{X_i}-\e s_{X_i}\right)
\longrightarrow\mathcal N(0,t_1)\text{ in distribution,}
\end{equation*}
where $\mathcal N(m,\sigma^2)$ is a normal real random variable with expectation $m$ and variance $\sigma^2.$ Note that $\e\left(s_{X_i^*}-\e s_{X_i^*}\right)=0$.\par

\bigskip\noindent
In the same way one can show that
\begin{multline*}
\left|\frac{1}{\tilde\sigma\sqrt n}\sum_{i=\lfloor nt_1\rfloor+1}^{\lfloor nt_2\rfloor}\left(s_{X_i^*}-\e s_{X_i^*}\right)-
\frac{\sqrt{t_2-t_1}}{\tilde\sigma(\lfloor nt_2\rfloor-\lfloor nt_1\rfloor-1)^{\frac{1}{2}}}
\sum_{i=\lfloor nt_1\rfloor+1}^{\lfloor nt_2\rfloor}\left(s_{X_i^*}-\e s_{X_i^*}\right)\right|\\
\longrightarrow 0\text{ in probability,}
\end{multline*}
and by the central limit theorem
\begin{equation*}
\frac{\sqrt{t_2-t_1}}{\tilde\sigma\sqrt{\lfloor nt_2\rfloor-\lfloor nt_1\rfloor -1}}\sum_{j=\lfloor nt_1\rfloor+1}^{\lfloor nt_2\rfloor}\left(s_{X_j^*}-\e s_{X_j^*}\right)
\longrightarrow\mathcal N(0,t_2-t_1)\text{ in distribution.}
\end{equation*}
Hence \eqref{Fourier} reads
\begin{equation}\label{ost}
\lim_{n\to\infty}F_{s_{\tilde S^*_{0,\lfloor nt_1\rfloor}},s_{\tilde S^*_{\lfloor nt_1\rfloor+1,\lfloor nt_2\rfloor}}}(u_1,u_2)
=e^{-u_1^2t_1\slash2}e^{-u_2^2(t_2-t_1)\slash2}.
\end{equation}
The right hand side in \eqref{ost} is the product of the Fourier transforms of variables $\mathcal N(0,t_1)$ and $\mathcal N(0,t_2-t_1),$ and the proof is finished.
\end{proof}

\end{document}